\documentclass[intlimits,12pt,reqno]{amsart}
\usepackage{latexsym,amsthm,amsmath,amssymb,mathrsfs,mathtools}
\usepackage[T1]{fontenc}
\usepackage[utf8]{inputenc}
\usepackage[mathscr]{eucal}
\usepackage{enumerate}
\usepackage{enumitem}

\usepackage{tikz}
\usepackage{pgfplots}
\pgfplotsset{compat=1.18}
\usetikzlibrary{intersections,backgrounds}

\let\oldmarginpar\marginpar
\renewcommand{\marginpar}[2][rectangle,draw,text width= 2cm,rounded corners]{
    \oldmarginpar{
    \scriptsize \tikz \node at (0,0) [#1]{#2};}
    }
\usepackage{tikz}
\usetikzlibrary{arrows,calc,patterns,cd}
\usepackage{wrapfig}

\usepackage{hyperref}
\usepackage{xcolor}
\hypersetup{
    colorlinks=true,
    linkcolor=blue,
    citecolor=blue,
    urlcolor=blue
}

plus 6pt minus 12pt
\def\mvint_#1{\mathchoice
          {\mathop{\vrule width 6pt height 3 pt depth -2.5pt
                  \kern -9pt \intop}\limits_{\kern -3pt #1}}%
          {\mathop{\vrule width 5pt height 3 pt depth -2.6pt
                  \kern -6pt \intop}\nolimits_{#1}}%
          {\mathop{\vrule width 5pt height 3 pt depth -2.6pt
                  \kern -6pt \intop}\nolimits_{#1}}%
          {\mathop{\vrule width 5pt height 3 pt depth -2.6pt
                  \kern -6pt \intop}\nolimits_{#1}}}

\newcommand{\bbbr}{\mathbb R}

\newcommand{\bbbn}{\mathbb N}

\newcommand{\R}{\mathbb R}

\newcommand{\eps}{\varepsilon}

\newtheorem{theorem}{Theorem}[section]
\newtheorem*{theorem*}{Theorem}
\newtheorem{lemma}[theorem]{Lemma}
\newtheorem{conjecture}[theorem]{Conjecture}
\newtheorem{corollary}[theorem]{Corollary}

\theoremstyle{definition}

\newtheorem*{remark*}{Remark}

\newcommand*{\loc}{{\mathrm{loc}}}

\usepackage{setspace}

\makeatletter
\renewcommand{\tocsection}[3]{%
  \indentlabel{\@ifnotempty{#2}{\bfseries\ignorespaces#1 #2\quad}}\bfseries#3}
\renewcommand{\tocsubsection}[3]{%
  \indentlabel{\@ifnotempty{#2}{\ignorespaces#1 #2\quad}}#3}

\newcommand\@dotsep{4.5}
\def\@tocline#1#2#3#4#5#6#7{\relax
  \ifnum #1>\c@tocdepth 
  \else
    \par \addpenalty\@secpenalty\addvspace{#2}%
    \begingroup \hyphenpenalty\@M
    \@ifempty{#4}{%
      \@tempdima\csname r@tocindent\number#1\endcsname\relax
    }{%
      \@tempdima#4\relax
    }%
    \parindent\z@ \leftskip#3\relax \advance\leftskip\@tempdima\relax
    \rightskip\@pnumwidth plus1em \parfillskip-\@pnumwidth
    #5\leavevmode\hskip-\@tempdima{#6}\nobreak
    \leaders\hbox{$\m@th\mkern \@dotsep mu\hbox{.}\mkern \@dotsep mu$}\hfill
    \nobreak
    \hbox to\@pnumwidth{\@tocpagenum{\ifnum#1=1\bfseries\fi#7}}\par
    \nobreak
    \endgroup
  \fi}
\AtBeginDocument{%
\expandafter\renewcommand\csname r@tocindent0\endcsname{0pt}
}
\def\l@subsection{\@tocline{2}{-5pt}{2.5pc}{5pc}{}}
\makeatother

\title[Approximation of convex functions]{$\mathbf{C^2}$-Lusin approximation of convex functions: one variable case}

\author[P. Goldstein]{Pawe\l{}  Goldstein}
\address{Pawe\l{} Goldstein, \newline \indent Institute of Mathematics, Faculty of Mathematics, Informatics and Mechanics, \newline \indent University of Warsaw, Banacha 2, 02-097 Warsaw, Poland} \email{P.Goldstein@mimuw.edu.pl}
\thanks{P.G.\ was supported by NCN grant no 2019/35/B/ST1/02030}

\author[Haj\l{}asz]{Piotr Haj\l{}asz}
\address{Piotr Haj\l{}asz,\newline \indent Department of Mathematics, University of Pittsburgh, \newline \indent 301 Thackeray Hall, Pittsburgh,
Pennsylvania 15260}
\email{hajlasz@pitt.edu}
\thanks{P.H. was supported by an NSF grant DMS-2452426}


\begin{document}

\begin{abstract}
We prove that if $f:(a,b)\to\R$ is convex, then for any $\eps>0$ there is a convex function $g\in C^2(a,b)$ such that $|\{f\neq g\}|<\eps$ and $\Vert f-g\Vert_\infty<\eps$.
\end{abstract}

\maketitle

\section{Introduction}
\label{intro}

It has been known for at least thirty years that convex functions have the {\em $C^2$-Lusin property}, meaning that if $f:\R^n\to\R$ is convex, then for every $\eps>0$ there is a function $g\in C^2(\R^n)$ such that $|\{f\neq g\}|<\eps$, see \cite{Alberti2,EvansGangbo,Imomkulov}. Here $|A|$ denotes the Lebesgue measure of $A$. In that case we say that $g$ is a {\em Lusin approximation of $f$}.

However, in general the function $g$ cannot be convex. For example for $f:\R^2\to\R$, $f(x,y)=|x|$, it is easy to show that if $g:\R^2\to\R$ is convex and $|\{f\neq g\}|<\infty$, then $f=g$ everywhere. In general we have
\begin{theorem}[\cite{ACH,AH}]
\label{T1}
Let $U\subset\R^n$ be open and convex, and let $f:U\to\R$ be a convex function, such that $f\not\in C^{1,1}_{\rm loc}(U)$. Then, the following statements are equivalent:
\begin{enumerate}
\item For every $\varepsilon>0$ there exists a convex function $g\in C^{1,1}_{\rm loc}(U)$  such that
$$
|\{x\in U : f(x)\neq g(x)\}|<\varepsilon.
$$
\item The graph of $f$ does not contain any line isometric to $\bbbr$.
\end{enumerate}
Moreover, if the graph of $f$ contains no lines, we can find $g$ satisfying $g\geq f$.
\end{theorem}
Here $C^{1,1}_{\rm loc}$ stands for the class of functions with locally Lipschitz continuous gradient.

However, the Lusin approximation by convex $C^{1,1}_{\loc}$ functions is much easier than the Lusin approximation by convex $C^2$ functions and the approximation by $C^2$ convex functions has only been proved under an additional condition that the function $f$ is locally strongly convex.

Recall that a function $f: U \to \R$ is {\em strongly convex}, where $U \subseteq \R^n$ is open and convex, if there exists $\eta > 0$ such that $f(x) - \frac{\eta}{2}|x|^2$ is convex. In this case, we say that $f$ is {\em $\eta$-strongly convex}. Moreover, $f$ is {\em locally strongly convex} whenever for every $x \in U$ there exists $r > 0$ such that the restriction of $f$ to $B(x, r) \subset U$ is strongly convex. Note that locally strongly convex functions cannot contain a line on the graph and thence they can always be approximated in the Lusin sense by convex $C^{1,1}_{\rm loc}$ functions. However, we have a much stronger result in that case:

\begin{theorem}[\cite{ADH}]
\label{T2}
Let $U\subseteq\R^n$ be open and convex, and let $f:U\to\bbbr$ be locally strongly convex. Then for every $\eps_o>0$ and for every continuous function $\varepsilon:U\to (0, 1]$ there is a locally strongly convex function
$g\in C^2(U)$, such that
$$|\{x\in U:\, f(x)\neq g(x)\}|<\eps_o
\quad
\text{and}
\quad
|f(x)-g(x)|<\varepsilon(x) \text{ for all } x\in U.
$$
Also, if $f$ is $\eta$-strongly convex on $U$, then for every $\widetilde{\eta}\in (0, \eta)$ there exists such a function $g$ which is $\widetilde{\eta}$-strongly convex on $U$.
\end{theorem}

The main difficulty in Theorems~\ref{T1} and~\ref{T2} stems from the fact that they are true in an arbitrary dimension. While the dimension $n=1$ is much easier to deal with, it seems it has not been fully investigated yet. The main result of the paper resolves the problem of Lusin approximation of convex functions when $n=1$.

\begin{theorem}
\label{T3}
Let $f:(a,b)\to\bbbr$ be convex, $-\infty\leq a<b\leq\infty$. Then for every $\eps_o>0$ and for every continuous function $\eps:(a,b)\to (0,\infty)$ there is a convex function $g\in C^2(a,b)$ such that
$$
|\{x\in (a,b):\, f(x)\neq g(x)\}|<\eps_o
\quad
\text{and}
\quad
|f(x)-g(x)|<\eps(x) \text{ for all } x\in (a,b).
$$
\end{theorem}
The uniform approximation described in the theorem allows one to have increasing accuracy when approaching the boundary. In particular, if $f$ is continuous on $[a,b]$, we can find $g\in C^2(a,b)\cap C([a,b])$ with $g(a)=f(a)$ and $g(b)=f(b)$.

The problem whether every convex function in dimension $n\geq 2$, whose graph does not contain a line, can be approximated in the Lusin sense by $C^2$ convex functions is open. However, we conjecture:

\begin{conjecture}
If $n\geq 2$, then there is a convex function $f:B^n(0,1)\to\R$ and $m>0$, such that if $g\in C^2(B^n(0,1))$ is convex, then
$|\{f\neq g\}|\geq m$.
\end{conjecture}

Let us explain now the main idea of the proof of Theorem~\ref{T3} along with the main difficulty. Given a convex function $f:(a,b)\to\bbbr$, we want to find a convex function $g\in C^2(a,b)$, such that $|\{f\neq g\}|<\eps$. Theorem~\ref{T1} allows us to assume that $f\in C^{1,1}_{\rm loc}$ and it is well known that a $C^{1,1}_{\rm loc}$ function coincides with a $C^2$ function outside a set of arbitrarily small measure. Thus, we can find $u\in C^2(a,b)$ such that $u=f$ on a closed set $E$ with $|(a,b)\setminus E|<\eps$. However, the function $u$ need not be convex and we want to correct it outside $E$ to make it $C^2$ and convex. The set $(a,b)\setminus E$ is the union of countably many intervals $(a_i,b_i)$. If $u$ is convex on $[a_i,b_i]$, we do not change it on $[a_i,b_i]$. If $u$ is not convex on $[a_i,b_i]$, we want to replace it with a convex $C^2$ function on $[a_i,b_i]$ while keeping values of $u$, $u'$ and $u''$ at the endpoints.

However, it is not always possible. If for example,
$$
f(x)=
\begin{cases}
x^2     & \text{if } x\leq 0,\\
0       & \text{if } 0\leq x\leq 1,\\
(x-1)^2 & \text{if } x\geq 1,
\end{cases}
$$
then $f$ is convex, of class $C^{1,1}$, but not of class $C^2$. Clearly, there is $u\in C^2(\bbbr)$, such that
$u=f$ on $\bbbr\setminus (0,1)$, but there is no convex $C^2(\bbbr)$ function that coincides with $f$ on $\bbbr\setminus (0,1)$. Indeed, it is easy to see that the only convex function that coincides with $f$ on $\bbbr\setminus (0,1)$ is $f$ itself. This is because the value of $f(1)$ lies on the tangent line at $x=0$.

Therefore, when we select the closed set $E\subset (a,b)$ where $u=f$, we have to make sure that the intervals $(a_i,b_i)$ have the property that the value of $f(b_i)$ lines above the tangent line at $a_i$ and that we have good estimates on how far $f(b_i)$ is from the tangent line at $a_i$, so that when we replace $u$ on infinitely many intervals $(a_i,b_i)$ with explicitly constructed convex functions of class $C^2([a_i,b_i])$, the resulting function will be of class $C^2(a,b)$.

\subsection*{Acknowledgements}
We are deeply grateful to Fedor Nazarov for generously sharing an idea that proved essential in completing our proof.

Piotr Hajłasz appreciates the hospitality of the University of Warsaw, where part of this work was conducted. His stay in Poland received funding from the
University of Warsaw via the IDUB project (Excellence Initiative Research University) as
part of the Thematic Research Programme \emph{Geometric Analysis: Methods and Applications GAMA25}. 

\section{Proof of Theorem~\ref{T3}}
\subsection{Proof of a special case}
In this section we prove a weaker result (Lemma~\ref{T9}). The general case is treated in Section~\ref{uni}.
\begin{lemma}
\label{T9}
Let $f:(a,b)\to\bbbr$ be convex, $a,b\in\bbbr$. Then for every $\eps>0$  there is a convex function $g\in C^2(a,b)$ such that
$
|\{x\in (a,b):\, f(x)\neq g(x)\}|<\eps.
$
\end{lemma}
Note that we assume here that the interval $(a,b)$ is bounded, while in Theorem~\ref{T3} we have any open interval.
\begin{proof}
According to Theorem~\ref{T1}, we can assume that $f\in C^{1,1}_{\rm loc}(a,b)$. Clearly, $f$ is twice differentiable a.e.

Fix small $\eps>0$.
It follows from Whitney's theorem \cite[Theorem~4]{whitney} that  $f\in C^{1,1}_{\rm loc}(a,b)$ coincides with a $C^2$ function $u\in C^2(a,b)$ outside a set of measure less that $\eps/3$. However, the result does not guarantee convexity of $u$.

Let $E_1\subset (a,b)$, $|(a,b)\setminus E_1|<\eps/3$, be a compact set such that $u=f$, $u'=f'$, $u''=f''$ on $E_1$.

Let $h:=f''$. According to the Lebesgue differentiation theorem, almost every $x\in (a,b)$ is a Lebesgue point of $h$ and at such  a point we have (see \cite[p.\ 11]{stein})
\begin{equation}
\label{eq3}
h_{r^{\pm}}(x)\stackrel{r\to 0}{\longrightarrow} h(x),
\quad
\text{where}
\quad
h_{r^+}(x):=\frac{1}{r}\int_{x}^{x+r} h(y)\, dy,
\quad
h_{r^-}(x):=\frac{1}{r}\int_{x-r}^{x} h(y)\, dy.
\end{equation}
Egorov's and Lusin's theorems yield a compact set $E_2\subset (a,b)$, $|(a,b)\setminus E_2|<\eps/3$, such that $h|_{E_2}$ is continuous and
\begin{equation}
\label{eq1}
h_{1/i^\pm}\rightrightarrows h
\quad
\text{uniformly on}
\quad
E_2 \text{ as } i\to\infty
\end{equation}
This implies that
\begin{equation}
\label{eq2}
h_{r^\pm}\rightrightarrows h
\quad
\text{uniformly on}
\quad
E_2 \text{ as } r\to 0.
\end{equation}
Indeed, $h$ is bounded in an open neighborhood of $E_2$, say $0\leq h\leq M$. For each $r\in (0,1)$, let $i\in \bbbn$ be such that $1/i\leq r<1/(i-1)$; it suffices to observe that for sufficiently small $r$,
$$
|h_{r^\pm}-h_{1/i^\pm}|\leq 2M/i\to 0
\quad
\text{as } i\to\infty.
$$
This and \eqref{eq1} imply \eqref{eq2}.

Fix $\delta>0$ such that
$
|\{x\in (a,b):\, h(x)\in (0,\delta)\}|<\eps/3,
$
and let
$$
E_3\subset \{x\in (a,b):\, h(x)\not\in (0,\delta)\},
\quad
|(a,b)\setminus E_3|<\eps/3,
$$
be a compact set. Clearly, $h(x)=0$ or $h(x)\geq \delta$ for all $x\in E_3$. Finally define
$$
E:=E_1\cap E_2\cap E_3,
\quad
\text{so}
\quad
|(a,b)\setminus E|<\eps.
$$
By removing countably many points from $E$, we can assume that there are no isolated points in $E$.

The function $u\in C^2(a,b)$ coincides with $f$ on $E$ and hence it approximates $f$ in the Lusin sense, i.e. $|\{u\neq f\}|<\eps$. However, $u$ is not necessarily convex. We will
construct a convex function $g\in C^2(a,b)$ such that $f=g$ on $E$ by applying an infinite  sequence of modifications to the function $u$ while keeping its values on the set $E$ unchanged.

The set $(a,b)\setminus E=\bigcup_i (a_i,b_i)$ is the union of countably many open intervals.
Since the set $E$ has no isolated points, any two intervals $(a_i,b_i)$ and $(a_j,b_j)$, $i\neq j$, have positive distance one from another.

Consider first the endpoint intervals $(a_i,b_i)=(a,b_i)$ and $(a_i,b_i)=(a_i,b)$.
Denote the endpoint intervals by $(a,\alpha):=(a,b_i)$ and $(\beta,b):=(a_i,b)$, so $E\subset I:=[\alpha,\beta]$.

We replace $u$ in $(a_i,b_i]=(a,\alpha]$ with any convex function $u_i\in C^2((a,\alpha])$ such that $u_i(\alpha)=u(\alpha)$, $u_i'(\alpha)=u'(\alpha)$, and $u_i''(\alpha)=u''(\alpha)$. We also replace $u$ on $[a_i,b_i)=[\beta,b)$ with $u_i\in C^2([\beta,b))$ in a similar way.
The resulting function will still be denoted by $u$.

Assume now that an interval $(a_i,b_i)$ is not an endpoint interval, i.e., $a_i>a$ and $b_i<b$.
Clearly, $[a_i,b_i]\subset I=[\alpha,\beta]$.
If
\begin{equation}
\label{eq5}
h(a_i)>0
\qquad
\text{or}
\qquad
h(b_i)>0,
\end{equation}
then there is a convex function $u_i\in C^2([a_i,b_i])$ such that
\begin{equation}
\label{eq9}
\begin{aligned}
&u_i(a_i)=u(a_i),\ u_i'(a_i)=u'(a_i),\ u_i''(a_i)=u''(a_i),\\
&u_i(b_i)=u(b_i),\ u_i'(b_i)=u'(b_i),\ u_i''(b_i)=u''(b_i).
\end{aligned}
\end{equation}
Indeed, if for example $h(a_i)>0$, then $f''=h>0$ on a subset of $(a_i,b_i)$ of positive measure, because \eqref{eq3} is satisfied at $x=a_i$. Therefore,
$$
u(b_i)=f(b_i)>f(a_i)+f'(a_i)(b_i-a_i)=u(a_i)+u'(a_i)(b_i-a_i),
$$
i.e., the value of $u(b_i)$ is above the tangent line at $a_i$ and this allows us to define $u_i$ explicitly by prescribing (piecewise linear) $u_i''\geq 0$ on $[a_i,b_i]$. For more details, see the discussion surrounding formula \eqref{eq11}.

We fix such a function $u_i$ for each interval $[a_i,b_i]$ satisfying \eqref{eq5}.

First, consider intervals $(a_i,b_i)$, such that
\begin{equation}
\label{eq14}
h(a_i)>0
\quad
\text{and}
\quad
h(b_i)>0.
\end{equation}
If $u$ is convex on $[a_i,b_i]$, we do not alter $u$ on $[a_i,b_i]$.  Thus, assume that
\begin{equation}
\label{eq8}
h(a_i)>0,
\quad
h(b_i)>0,
\quad
u \text{ is not convex on } [a_i,b_i].
\end{equation}

Uniform continuity of $u''$ on $I=[\alpha,\beta]$ implies that $I$ contains at most finitely many intervals $[a_i,b_i]$ satisfying \eqref{eq8} (they are pairwise disjoint, because $E$ has no isolated points). Indeed, there is $\tau>0$, such that
$$
|x-y|<\tau,\,\,\,\, x,y\in I
\quad
\Longrightarrow
\quad
|u''(x)-u''(y)|<\delta.
$$
If $u$ is not convex on $[a_i,b_i]$, then there is $x\in (a_i,b_i)$ with $u''(x)<0$. Hence
$$
|u''(a_i)-u''(x)|>u''(a_i)=h(a_i)\geq\delta
\quad
\text{implies that}
\quad
|a_i-b_i|>|a_i-x|\geq\tau.
$$
Clearly, we can have only a finite number of pairwise disjoint intervals of length at least $\tau$ inside $I$.
On intervals satisfying \eqref{eq8} we redefine $u$ as follows:
$$
u(x):=u_i(x)
\quad
\text{if } x\in [a_i,b_i]
\text{ and } \eqref{eq8} \text{ is satisfied.}
$$
This completes the redefinition of $u$ on all intervals satisfying \eqref{eq14}.
Note that although we redefined the function $u$ on intervals $[a_i,b_i]$, we still denote the new function by $u$.

Now, let us consider intervals satisfying
\begin{equation}
\label{eq6}
\big(h(a_i)=0 \text{ and } h(b_i)>0\big)
\qquad
\text{or}
\qquad
\big(h(a_i)>0 \text{ and } h(b_i)=0\big).
\end{equation}

Again, the interval $I=[\alpha,\beta]$ contains at most finitely many intervals $[a_i,b_i]\subset I$ satisfying \eqref{eq6}.
Indeed, since $h$ is uniformly continuous on $E$, there is a constant $\tau>0$, such that
$$
|x-y|<\tau,\ x,y\in E
\quad
\Longrightarrow
\quad
|h(x)-h(y)|<\delta.
$$
Since \eqref{eq6} implies that $|h(a_i)-h(b_i)|\geq \delta$, it follows that $|a_i-b_i|\geq\tau$,
and we can only have a finite number of such intervals in $I$.
As in the previous case, on such intervals we replace $u$ with
\begin{equation}
\label{eq7}
u(x):=u_i(x)
\quad
\text{if } x\in [a_i,b_i] \text{ and } \eqref{eq6} \text{ is satisfied.}
\end{equation}

We corrected $u$ on all intervals $[a_i,b_i]$ satisfying \eqref{eq14} or \eqref{eq6}, that is, on all intervals satisfying \eqref{eq5}, on which $u$ was not convex already. Since we altered $u$ on a finite family of pairwise disjoint intervals while keeping the values of $u$, $u'$ and $u''$ at the endpoints, the resulting new function (still denoted by $u$) is of class $C^2(a,b)$ and
\begin{equation}
\label{eq20}
u''(x)\geq 0 \text{ whenever } x\not\in (a_i,b_i)
\text{ for some }  i \text{ such that } h(a_i)=h(b_i)=0.
\end{equation}
The remaining intervals on which we need to modify the function $u$ are such that
\begin{equation}
\label{eq10}
h(a_i)=h(b_i)=0.
\end{equation}
Denote by $\mathcal{I}$ the set of indices $i$ satisfying \eqref{eq10}.

We will replace $u$ on the intervals $[a_i,b_i]$, $i\in\mathcal{I}$, with suitable $C^2$ convex functions ${u_i\in C^2([a_i,b_i])}$ satisfying \eqref{eq9}.
In the previous modifications of $u$, we altered it on a finite family of pairwise disjoint intervals clearly preserving the $C^2$ regularity of $u$. However, this time the family of intervals $[a_i,b_i]$, $i\in\mathcal{I}$, may be infinite and that the resulting modified function is still of class $C^2(a,b)$ will require some work.

Since $u$, $u'$ and $u''$ agree with $f$, $f'$ and $f''$ at the points $a_i$ and $b_i$, we are looking for a $C^2$ convex function $u_i$ such that $u_i$, $u'_i$ and $u''_i$ agree with $f$, $f'$ and $f''$ at the points $a_i$ and $b_i$ (in particular, $u_i''(a_i)=u_i''(b_i)=0)$.

The simplest case is when $h=0$ at almost every point of $(a_i,b_i)$. Then the function $f$ is linear on $[a_i,b_i]$ and hence $u'(a_i)=u'(b_i)$. In that case we define $u_i$ to be linear on $[a_i,b_i]$.

Therefore we may assume that $\int_{a_i}^{b_i}h(t)\, dt>0$. Upon translation, we can assume that $[a_i,b_i]=[0,c_i]$, where $c_i=b_i-a_i$. This will slightly simplify notation.

Since $f\in C^{1,1}([0,c_i])$ and $h=f''$, Taylor's formula with the integral remainder yields
$$
f'(c_i)=f'(0)+\int_0^{c_i} h(t)\, dt,
\quad
f(c_i)=f(0)+f'(0)c_i+\int_0^{c_i} h(t)(c_i-t)\, dt.
$$
Therefore, if $h_i\in C([0,c_i])$ satisfies
\begin{equation}
\label{eq11}
h_i(0)=h_i(c_i)=0,
\quad
\int_0^{c_i} h_i(t)\, dt=\int_0^{c_i} h(t)\, dt,
\quad
\int_0^{c_i} t h_i(t)\, dt=\int_0^{c_i} t h(t)\, dt,
\end{equation}
and $u_i$ is defined as the second anti-derivative of $h_i$ i.e,
\begin{equation}
\label{eq15}
u_i(x):=f(0)+f'(0)x+\int_0^x\int_0^\tau h_i(t)\, dt\, d\tau\stackrel{{\rm(Fubini)}}{=}
f(0)+f'(0)x+\int_0^x h_i(t)(x-t)\, dt,
\end{equation}
we will have that $u_i$, $u'_i$ and $u''_i$ agree with $f$, $f'$ and $f''$ at the points $0$ and $c_i$ (i.e. upon translation, at the points $a_i$ and $b_i$).
Thus we need to construct a function $h_i$ satisfying \eqref{eq11}, and we need to do it carefully, because we need to get good estimates for the resulting function $u_i$.

For $i\in\mathcal{I}$, let $P_i$ and $\tau_i$ be defined by
$$
P_i:=\int_0^{c_i} h(t)\, dt,
\quad
P_i\tau_i=\int_0^{c_i}th(t)\, dt.
$$
Clearly, $\tau_i\in (0,c_i)$.

If $\tau_i\in (0,c_i/2]$, we define $H_i:=P_i/\tau_i$ and $h_i:[0,c_i]\to [0,\infty)$ by
\begin{equation}
\label{eq16}
h_i(t):=
\begin{cases}
H_i\tau^{-1}_it          & \text{if } t\in [0,\tau_i],\\
-H_i\tau_i^{-1}(t-2\tau_i)     & \text{if } t\in [\tau_i,2\tau_i],\\
0                        & \text{if } t\in [2\tau_i,c_i].
\end{cases}
\end{equation}
If $\tau_i\in [c_i/2,c_i)$, we define $H_i:=P_i/(c_i-\tau_i)$ and $h_i:[0,c_i]\to [0,\infty)$ by
\begin{equation}
\label{eq17}
h_i(t):=
\begin{cases}
0                        & \text{if } t\in [0,2\tau_i-c_i],\\
H_i(c_i-\tau_i)^{-1}(t+c_i-2\tau_i)     & \text{if } t\in [2\tau_i-c_i,\tau_i],\\
-H_i(c_i-\tau_i)^{-1}(t-c_i)                        & \text{if } t\in [\tau_i,c_i].
\end{cases}
\end{equation}
It is helpful to sketch the graph of $h_i$.
In geometric terms, the graph of $h_i$ is an isosceles triangle of height $H_i$ and the area equal $P_i$. In the case when $\tau_i\in (0,c_i/2]$, the base  of the triangle touches zero and in the case when $\tau_i\in [c_i/2,c_i)$ the base touches $c_i$.
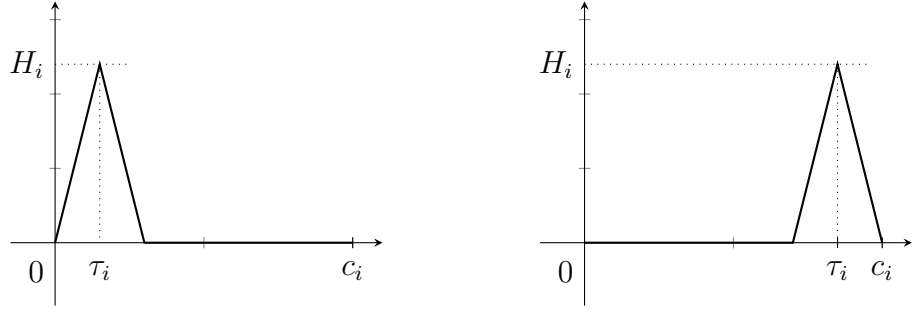
\begin{figure}[!h]
\begin{tikzpicture}
  \begin{axis}[
    domain=-0.6:5,
    samples=100,
    width=6.5cm,
    axis lines=center,
    axis equal,
    xticklabels=\empty,
yticklabels=\empty,
  ]
  \addplot[mark=none,mark size=0.5pt, black] coordinates {(2.2,0)};
  \addplot[mark=none,mark size=0.5pt, black] coordinates {(-0.3,0)};
    \addplot[thick] coordinates {(0,0) (0.3,1.2) (0.6,0) (2,0)};
    \addplot[dotted] coordinates {(0,1.2) (0.5,1.2)};
    \addplot[dotted] coordinates {(0.3,1.2) (0.3,0)};
    \addplot[mark=|, mark size=2pt, black] coordinates {(2,0)};
    \node at (axis cs:2,-0.05) [anchor=north] {$c_i$};
    \node at (axis cs:0.3,-0.05) [anchor=north] {$\tau_i$};
    \node at (axis cs:0,1.2) [anchor=east] {$H_i$};
    \node at (axis cs:0,-0.05) [anchor=north east] {$0$};

      \end{axis}
    \begin{axis}[xshift=7cm,
    domain=-0.6:5,
    samples=100,
    width=6.5cm,
    axis lines=center,
    axis equal,
    xticklabels=\empty,
yticklabels=\empty,
  ]
  \addplot[mark=none,mark size=0.5pt, black] coordinates {(2.2,0)};
  \addplot[mark=none,mark size=0.5pt, black] coordinates {(-0.3,0)};
    \addplot[thick] coordinates {(0,0) (1.4,0) (1.7,1.2) (2,0)};
    \addplot[dotted] coordinates {(0,1.2) (1.9,1.2)};
    \addplot[dotted] coordinates {(1.7,1.2) (1.7,0)};
    \addplot[mark=|, mark size=2pt, black] coordinates {(1.7,0)};
      \addplot[mark=|, mark size=2pt, black] coordinates {(2,0)};
    \node at (axis cs:2,-0.05) [anchor=north] {$c_i$};
    \node at (axis cs:1.7,-0.05) [anchor=north] {$\tau_i$};
    \node at (axis cs:0,1.2) [anchor=east] {$H_i$};
    \node at (axis cs:0,-0.05) [anchor=north east] {$0$};

      \end{axis}
\end{tikzpicture}
\caption{The graph of $h_i$  when $\tau_i\in (0,c_i/2]$ (left) and $\tau_i\in [c_i/2, c_i)$ (right).}
\end{figure}

It is easy to check that
$$
\int_0^{c_i} h_i(t)\, dt=P_i=\int_0^{c_i} h(t)\, dt,
\quad
\int_0^{c_i}th_i(t)\, dt =P_i\tau_i=\int_0^{c_i} th(t)\, dt.
$$
One can check it by computing the integrals directly, but the computations are annoying, or one can argue as follows.
Clearly $\int_0^{c_i} h_i(t)\, dt$ equals the area of the triangle, which is $P_i$. If $h_i$ is regarded as a mass density, the center of mass is located at $\tau_i$ (by symmetry) and hence
$$
\frac{\int_0^{c_i} th_i(t)\, dt}{\int_0^{c_i} h_i(t)\, dt}=\tau_i,
\quad
\int_0^{c_i}th_i(t)\, dt=P_i\tau_i.
$$
Therefore if $i\in\mathcal{I}$ and $\int_{a_i}^{b_i} h(t)\, dt>0$, we define convex $u_i\in C^2([0,c_i])$ and after translation, convex $u_i\in C^2([a_i,b_i])$ using \eqref{eq15}, \eqref{eq16} and \eqref{eq17}.

Until now we did not make any distinction between the two cases when $\mathcal{I}$ is a finite and an infinite set. The case of finite $\mathcal{I}$ is easy: the procedure described above consists of a finite number of corrections on disjoint intervals $[a_i,b_i]$; the $C^2$-differentiability of the resulting function is immediate. The case of infinite $\mathcal{I}$ requires an additional argument. 

The next result provides a crucial estimate for the functions $u_i''$, $i\in \mathcal{I}$. Let
\begin{equation}
\label{eq13}
\eps_i:=
\max\Big\{\sup_{x\in (0,c_i]}\frac{1}{x}\int_0^x h(t)\, dt,\,\, \sup_{x\in (0,c_i]}\frac{1}{x}\int_{c_i-x}^{c_i} h(t)\, dt \Big\}\, .
\end{equation}

\begin{lemma}
\label{T5}
Let $i\in\mathcal{I}$. Then $0\leq u''_i\leq 4\eps_i$ on $[a_i,b_i]$.
\end{lemma}
\begin{proof}
If $\int_{a_i}^{b_i} h(t)\, dt=0$, $u_i$ is linear and hence $u_i''=0$. Thus we may assume that the integral is positive. Since $u_i''=h_i$ and $0\leq h_i\leq H_i$, it suffices to prove that $H_i\leq 4\eps_i$.
First, assume that $\tau_i\in (0,c_i/2]$. For $x\in (0,c_i)$ we have
\begin{equation}
\label{eq12}
P_i=\int_0^{c_i} h(t)\, dt\leq
x\cdot\frac{1}{x}\int_0^x h(t)\, dt + \frac{1}{x}\int_x^{c_i} th(t)\, dt\leq
x\eps_i+\frac{1}{x}P_i\tau_i.
\end{equation}
Since $P_i\leq c_i\eps_i$, we have $x:=P_i/(2\eps_i)\leq c_i/2$ and hence we can apply it to \eqref{eq12} which readily yields $H_i=P_i/\tau_i\leq 4\eps_i$.

Now, assume that $\tau_i\in [c_i/2,c_i)$. For $x\in (0,c_i)$ we have
$$
P_i\leq\frac{1}{c_i-x}\int_0^x(c_i-t) h(t)\, dt+(c_i-x)\cdot\frac{1}{c_i-x}\int_x^{c_i} h(t)\, dt\leq
\frac{(c_i-\tau_i)P_i}{c_i-x} +(c_i-x)\eps_i,
$$
and taking $x=c_i-P_i/(2\eps_i)\in [c_i/2,c_i)$ yields $H_i=P_i/(c_i-\tau_i)\leq 4\eps_i$.
\end{proof}

\begin{corollary}
\label{T6}
Assume the set $\mathcal{I}$ is infinite. Then we have
\begin{equation}
\label{eq19}
\lim_{\mathcal{I}\ni i\to\infty} \sup_{[a_i,b_i]} |u_i''|=\lim_{\mathcal{I}\ni i\to\infty} \sup_{[a_i,b_i]} |u''|=0,
\end{equation}
and hence
\begin{equation}
\label{eq18}
\lim_{\mathcal{I}\ni i\to\infty} \sup_{[a_i,b_i]}\big(|u_i-u|+|(u_i-u)'|+|(u_i-u)''|\big)=0.
\end{equation}
\end{corollary}
\begin{proof}
Note that $\lim_{\mathcal{I}\ni i\to\infty} (b_i-a_i)=0$. Hence
$\lim_{\mathcal{I}\ni i\to\infty} \sup_{[a_i,b_i]} |u''|=0$ follows from the uniform continuity of $u''$ on $I=[\alpha,\beta]$ and from the fact that $u''(a_i)=u''(b_i)=0$.
Also, $\lim_{\mathcal{I}\ni i\to\infty} \eps_i=0$. Indeed, convergence of $\eps_i$ to zero follows from the uniform convergence \eqref{eq2}, the definition \eqref{eq13} of $\eps_i$ and the fact that $h(a_i)=h(b_i)=0$. Therefore,
$\lim_{\mathcal{I}\ni i\to\infty} \sup_{[a_i,b_i]} |u_i''|=0$ is a consequence of Lemma~\ref{T5}. This proves \eqref{eq19}.

Clearly, \eqref{eq19} implies that $\lim_{\mathcal{I}\ni i\to\infty} \sup_{[a_i,b_i]}|(u_i-u)''|=0$. Since $u_i(a_i)-u(a_i)=0$ and $u_i'(a_i)-u'(a_i)=0$, \eqref{eq18} follows upon integration.
\end{proof}
Now, we are ready to make the final modification of the function $u$.
Recall that the function $u$ has already been modified on intervals satisfying \eqref{eq5}.

We define
$$
g(x):=
\begin{cases}
u_i(x) & \text{if } x\in [a_i,b_i],\ i\in \mathcal{I},\\
u(x)   & \text{if } x\in (a,b)\setminus\bigcup_{i\in\mathcal{I}} [a_i,b_i].
\end{cases}
$$
\begin{lemma}
\label{T8}
$g\in C^2(a,b)$ is convex.
\end{lemma}
\begin{proof}
This is obvious when $\mathcal{I}$ is finite, since by construction $u_i=u$, $u_i'=u'$, $u_i''=u''$ at the endpoints $a_i$ and $b_i$. Assume thus that $\mathcal{I}$ is infinite.
For $i\in \mathcal{I}$ let $v_i:(a,b)\to\bbbr$ be defined by
$$
v_i(x)=
\begin{cases}
u_i(x)-u(x) & \text{if } x\in [a_i,b_i],\\
0           & \text{if } x\in (a,b)\setminus [a_i,b_i].
\end{cases}
$$
Again, since $u_i=u$, $u_i'=u'$, $u_i''=u''$ at the endpoints $a_i$ and $b_i$, it follows that $v_i\in C^2(a,b)$. Therefore, the uniform convergence \eqref{eq18} yields that
$$
\sum_{i\in \mathcal{I}} v_i\in C^2(a,b),
$$
and hence, $g=u+\sum_{i\in\mathcal{I}}v_i\in C^2(a,b)$. Convexity of $g$ follows then from \eqref{eq20} and from the convexity of the functions $u_i$. The proof of the Lemma~\ref{T8} is complete.
\end{proof}
Since $f=g$ on $E$, we have that $|\{f\neq g\}|<\eps$ and this completes the proof of Lemma~\ref{T9}.
\end{proof}
\subsection{Proof in the general case}
\label{uni}
In his section, we will use Lemma~\ref{T9} to prove the general case of Theorem~\ref{T3}.
We will precede the proof with a couple of technical lemmata.
\begin{lemma}
\label{T11}
Suppose that $-\infty<\alpha<\beta<\gamma<\delta<\infty$, $f:[\alpha,\delta]\to\bbbr$ is convex, and $f|_{[\alpha,\beta]}\in C^2([\alpha,\beta])$, $f|_{[\gamma,\delta]}\in C^2([\gamma,\delta])$. Then for any $\eps>0$, there is a convex function $g\in C^2([\alpha,\delta])$, such that
$$
f=g
\quad
\text{on}
\quad
[\alpha,\beta-\eps]\cup [\gamma+\eps,\delta].
$$
\end{lemma}
\begin{proof}
If
\begin{equation}
\label{eq21}
f(\gamma+\eps)> f(\beta-\eps)+f'(\beta-\eps)\big((\gamma+\eps)-(\beta-\eps)\big),
\end{equation}
i.e., if the value of $f(\gamma+\eps)$ lies above the tangent line at $\beta-\eps$, we have no problem with finding $g$. If, however, we have equality in \eqref{eq21} instead of inequality, then $f$ is linear on $[\beta-\eps,\gamma+\eps]$ and it follows that $f\in C^2([\alpha,\delta])$, so in that case we take $g=f$.
\end{proof}

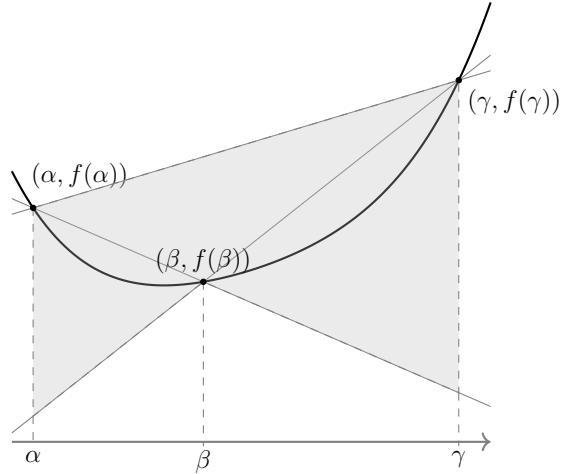
\begin{figure}[ht]\label{fig:2}
\begin{tikzpicture}
  \begin{axis}[
    domain=-1.1:1.15,
    samples=100,
    width=10cm,
    axis lines=none,
    axis equal,
  ]
    \addplot[thick] {0.3*(x^4 + x^2+x)};
    \addplot[mark=*, mark size=1pt, black] coordinates {(-1,0.3)};
    \node[scale=0.8] at (axis cs:-1.05,0.35) [anchor=south west] {$(\alpha, f(\alpha))$};
    \addplot[mark=*, mark size=1pt, black] coordinates {(1,0.9)};
     \node[scale=0.8] at (axis cs:1,0.9) [anchor=north west] {$(\gamma, f(\gamma))$};
    
    \addplot[gray] {0.3*x+0.6};
    \addplot[gray] {-0.4344*x-0.1344};
    \addplot[gray] {0.7896*x+0.1104};
    \addplot[
  fill=lightgray,
  fill opacity=0.3,
  draw=gray,
  dashed
] coordinates {
  (-1, 0.3)
  (-1, -0.6792)
  (-0.2, -0.04752)
  (1, -0.5688)
  (1, 0.9)
  (-1, 0.3) 
};
\addplot[gray,thick,->] {-0.8};
\addplot[draw=gray,dashed] coordinates {(-1, -0.6792)(-1, -0.82)};
\node[scale=0.8] at (axis cs:-1,-0.8) [anchor=north] {$\alpha$};
\addplot[draw=gray,dashed] coordinates {(-0.2, -0.04752)(-0.2, -0.82)};
\node[scale=0.8] at (axis cs:-0.2,-0.8) [anchor=north] {$\beta$};
\addplot[draw=gray,dashed] coordinates {(1, -0.5688)(1, -0.82)};
\node[scale=0.8] at (axis cs:1,-0.8) [anchor=north] {$\gamma$};
\addplot[mark=*, mark size=1pt, black] coordinates {(-0.2,-0.04752)};
     \node[scale=0.8] at (axis cs:-0.2,-0.04752) [anchor=south] {$(\beta, f(\beta))$};
      \end{axis}
\end{tikzpicture}\caption{Lemma \ref{T10}: the graph of $g$ must be contained in the shaded area.}
\end{figure}
\begin{lemma}
\label{T10}
Assume that $-\infty<\alpha<\beta<\gamma<\infty$ and that $f:[\alpha,\gamma]\to\bbbr$ is convex and $L$-Lipschitz. If $g:[\alpha,\gamma]\to\bbbr$ is convex and $f(\alpha)=g(\alpha)$, $f(\beta)=g(\beta)$, $f(\gamma)=g(\gamma)$, then
$$
\sup_{x\in [\alpha,\gamma]} |f(x)-g(x)|\leq 2L|\gamma-\alpha|.
$$
\end{lemma}
\begin{proof}
Let $x\in [\alpha,\gamma]$. Then $x\in[\alpha,\beta]$ or $x\in[\beta,\gamma]$. Assume that $x\in[\alpha,\beta]$; the estimate in the other case is similar.
It follows from the convexity of the functions $f$ and $g$ that both values $f(x)$ and $g(x)$ are squeezed between the line passing through $(\alpha,f(\alpha))$ and $(\gamma,f(\gamma))$ and the line passing through $(\beta,f(\beta))$ and $(\gamma,f(\gamma))$, that is
$$
f(\gamma)+\frac{f(\gamma)-f(\beta)}{\gamma-\beta}(x-\gamma)\leq \ f(x),\  g(x)\ \leq
f(\gamma)+\frac{f(\gamma)-f(\alpha)}{\gamma-\alpha}(x-\gamma),
$$
see Figure \ref{fig:2}.
Therefore,
$$
|f(x)-g(x)|\leq\Big(\frac{f(\gamma)-f(\alpha)}{\gamma-\alpha}-\frac{f(\gamma)-f(\beta)}{\gamma-\beta}\Big)(x-\gamma)
\leq 2L|\gamma-\alpha|.
$$

\end{proof}

\begin{proof}[Proof of Theorem~\ref{T3}]
Let $\eps_o>0$ and continuous $\eps:(a,b)\to (0,\infty)$ be given. We write
$$
(a,b)=\bigcup_{i=-\infty}^\infty [\alpha_i,\alpha_{i+1}],
\qquad
\alpha_i<\alpha_{i+1}.
$$
Using Lemma~\ref{T9}, we can approximate $f$ on $(\alpha_i,\alpha_{i+1})$ in the Lusin sense by a convex function $g_i\in C^2(\alpha_i,\alpha_{i+1})$ so that the measure
\begin{equation}
\label{eq22}
|\{x\in (\alpha_i,\alpha_{i+1}):\, f(x)\neq g_i(x)\}|
\end{equation}
is as small as we wish. Then, Lemma~\ref{T11} allows us to glue $g_i$'s restricted to slightly shorter compact intervals $I_i\subset (\alpha_i,\alpha_{i+1})$ to a convex function $g\in C^2(a,b)$.
It is important that $g_i$ and $g_i'$ agrees with $f$ and $f'$ at the endpoints of the intervals $I_i$. This will guarantee that $g_i$ on $I_i$ followed by $f$ on the interval between $I_i$ and $I_{i+1}$ (for every $i$), will be convex on $(a,b)$ -- this is needed if we want to apply Lemma~\ref{T11}.

Making sure that the measure \eqref{eq22} is sufficiently small and that
$|(\alpha_i,\alpha_{i+1})\setminus I_i|$ is sufficiently small, we can guarantee that
$$
|\{x\in (a,b):\, f(x)\neq g(x)\}|<\eps_o,
$$
and that
\begin{equation}
\label{eq23}
|\{x\in [\alpha_{i-1},\alpha_{i+2}]: f(x)\neq g(x)\}|<
\min\Big\{\frac{\inf_{t\in [\alpha_{i-1},\alpha_{i+2}]}\eps(t)}{4 L_i},
\alpha_i-\alpha_{i-1}, \alpha_{i+2}-\alpha_{i+1}\Big\},
\end{equation}
where $L_i$ is the Lipschitz constant of $f$ restricted to $[\alpha_{i-1},\alpha_{i+2}]$ (recall that a convex function defined on an open set is always locally Lipschitz \cite[Theorem 3.1.2]{HiriartUrruty}).
Let $\xi_i$ be a number between the numbers on the left-hand side and the right-hand side in \eqref{eq23}.

For any $x\in [\alpha_i,\alpha_{i+1}]$, the interval $[x-\xi_i,x]\subset [\alpha_{i-1},\alpha_{i+2}]$ contains a point where $f=g$ (a set of positive measure, in fact) and the interval $[x,x+\xi_i]\subset [\alpha_{i-1},\alpha_{i+2}]$ contains two points where $f=g$.
Since
$$
|(x+\xi_i)-(x-\xi_i)|=2\xi_i<\frac{\inf_{t\in [\alpha_{i-1},\alpha_{i+2}]}\eps(t)}{2 L_i},
$$
Lemma~\ref{T10} implies that
$$
|f(x)-g(x)|\leq 2L_i|(x+\xi_i)-(x-\xi_i)|<\inf_{t\in [\alpha_{i-1},\alpha_{i+2}]}\eps(t)\leq \eps(x).
$$
This completes the proof of the theorem.
\end{proof}

\end{document}